\documentclass[oneside,11pt,reqno]{amsart}

\usepackage[a4paper,margin=30mm]{geometry}

\usepackage{verbatim,upref,amsxtra,amssymb,amscd}

\DeclareMathAlphabet\mathscr{U}{eus}{m}{n}
\SetMathAlphabet\mathscr{bold}{U}{eus}{b}{n}
\DeclareMathAlphabet\matheur{U}{eur}{m}{n}
\SetMathAlphabet\matheur{bold}{U}{eur}{b}{n}

\numberwithin{equation}{section}

\newtheorem{theo}{Theorem}[section]
\newtheorem{prop}[theo]{Proposition}

\newtheorem{coro}[theo]{Corollary}

\theoremstyle{definition}

\newtheorem{defi}[theo]{Definition}
\newtheorem{exam}[theo]{Example}
\newtheorem{exas}[theo]{Examples}

\theoremstyle{remark}

\newtheorem{rema}[theo]{Remark}
\newtheorem{rems}[theo]{Remarks}

\begin{document}\allowdisplaybreaks\frenchspacing

\setlength{\baselineskip}{1.1\baselineskip}

\title{Algebraic Polymorphisms}

\author{Klaus Schmidt}

\address{Klaus Schmidt: Mathematics Institute, University of Vienna, Nordberg\-stra{\ss}e 15, A-1090 Vienna, Austria \newline\indent \textup{and} \newline\indent Erwin Schr\"odinger Institute for Mathematical Physics, Boltzmanngasse~9, A-1090 Vienna, Austria} \email{klaus.schmidt@univie.ac.at}

\author{Anatoly Vershik}

\address{Anatoly Vershik: Steklov Institute of Mathematics at St. Petersburg, 27 Fontanka, St.Petersburg 191011, Russia} \email{vershik@pdmi.ras.ru}

\thanks{The first author would like to thank the Department of Mathematics, Ohio State University, Columbus, Ohio, for hospitality while some of this work was done. The second author would like to thank the Erwin Schr\"{o}dinger Institute in Vienna for hospitality and support while some of this work was done, and gratefully acknowledges partial support by the grants RFBR 05-01-0089 and NSh-432920061}

\dedicatory{Dedicated to the memory of our colleague and friend William Parry}


    \begin{abstract}
In this paper we consider a special class of polymorphisms with invariant measure, the algebraic polymorphisms of compact groups. A general polymorphism is --- by definition --- a many-valued map with invariant measure, and the conjugate operator of a polymorphism is a Markov operator (i.e., a positive operator on $L^2$ of norm 1 which preserves the constants). In the algebraic case a polymorphism is a correspondence in the sense of algebraic geometry, but here we investigate it from a dynamical point of view.
The most important examples are the algebraic polymorphisms of torus, where we introduce a parametrization of the semigroup of toral polymorphisms in terms of rational matrices and describe the spectra of the corresponding Markov operators.

A toral polymorphism is an automorphism of $\mathbb{T}^m$ if and only if the associated rational matrix lies in $\textup{GL}(m,\mathbb{Z})$. We characterize toral polymorphisms which are factors of toral automorphisms.
    \end{abstract}

\maketitle

\section{Algebraic polymorphisms}\label{s:poly}

    \begin{defi}
    \label{d:algebraic}
Let $G$ be a compact group with Borel field $\mathscr{B}_G$, normalized Haar measure $\lambda _G$ and identity element $1=1_G$. A closed subgroup $\mathsf{P}\subset G\times G$ is an (\emph{algebraic}) \emph{correspondence of $G$} if $\pi _1(\mathsf{P})=\pi _2(\mathsf{P})=G$, where $\pi _i\colon G\times G\longrightarrow G$, $i=1,2$, are the coordinate projections (which are obviously group homomorphisms).

Every correspondence $\mathsf{P}\subset G\times G$ defines a map $\Pi _\mathsf{P}$ from $G$ to the set of all nonempty closed subsets of $G$ by
    \begin{equation}
    \label{eq:poly}
\Pi _\mathsf{P}(x)=\{y:(x,y)\in \mathsf{P}\}
    \end{equation}
for every $x\in G$. Clearly, $\pi _i$ sends the Haar measure on $\mathsf{P}$ to Haar measure on $G$; in the terminology of \cite{Ve}, the correspondence $\mathsf{P}$ defines an (\emph{algebraic}) \emph{polymorphism} of $G$ (more exactly, $\mathsf{P}$ determines a \emph{polymorphism of the measure space $(G,\mathscr{B}_G,\lambda _G)$ to itself}).\footnote{In general, a measure-preserving polymorphism $\Pi $ of a probability space $(X,\mathscr{S},\mu )$ is determined by a probability measure $\nu $ on $X\times X$ with ${\pi _i}_*\nu =\mu $ for $i=1,2$, i.e., by a \textit{coupling} of $\mu $ with itself.}

The correspondence $\mathsf{P}\subset G\times G$ and the polymorphism $\Pi _\mathsf{P}$ obviously determine each other.

Algebraic polymorphisms from one compact group to another are defined similarly.

A correspondence $\mathsf{P}'\subset G'\times G'$ is a \textit{factor} of a correspondence $\mathsf{P}\subset G\times G$ (and the polymorphism $\Pi _{\mathsf{P}'}$ is a factor of $\Pi _\mathsf{P}$) are \textit{isomorphic} if there exists a surjective group homomorphism $\phi \colon G\longrightarrow G'$ with $(\phi \times \phi )(\mathsf{P})=\mathsf{P}'$. If $\phi $ can be chosen to be a group isomorphism then $\mathsf{P}$ and $\mathsf{P}'$ (resp. $\Pi _\mathsf{P}$ and $\Pi _{\mathsf{P}'}$) are \textit{isomorphic}.

This notion of factors is consistent with the terminology in \cite{Ve}: if $\Pi $ is a measure-preserving polymorphism of a probability space $(X,\mathscr{S},\mu )$ determined by a self-coupling $\nu $ of $\mu $, and if $\mathscr{T}\subset \mathscr{S}$ is a sub-sigma-algebra, then the factor polymorphism $\Pi _\mathscr{T}$ of $(X,\mathscr{T})$ is determined by the restriction of $\nu $ to the sigma-algebra $\mathscr{T}\otimes\mathscr{T}\subset \mathscr{S}\otimes\mathscr{S}$.
    \end{defi}

\medskip Let $\mathsf{P}\subset G\times G$ be a correspondence (since we only consider algebraic correspondences and polymorphisms we drop the term \emph{algebraic} from now on). The subgroup
    \begin{equation}
    \label{eq:inverse}
\mathsf{P}^*=\{(y,x):(x,y)\in \mathsf{P}\}
    \end{equation}
corresponds to the \emph{conjugate} (or \emph{inverse}) polymorphism of $\Pi _\mathsf{P}$. If $\mathsf{P}_1,\mathsf{P}_2$ are two correspondences of $G$, their \emph{product} $\mathsf{P}_1\star\mathsf{P}_2$ is the correspondence
    \begin{equation}
    \label{eq:product}
    \begin{aligned}
\mathsf{P}_1\star\mathsf{P}_2=\{(x,z)\in G\times G &:(x,y)\in \mathsf{P}_2\enspace \textup{and}\enspace (y,z)\in \mathsf{P}_1
    \\
&\enspace \textup{for at least one}\enspace y\in G\}.
    \end{aligned}
    \end{equation}
Clearly,
    $$
\Pi _{\mathsf{P}_1\star \mathsf{P}_2}(x)=\Pi _{\mathsf{P}_1}\circ \Pi _{\mathsf{P}_2}(x)=\bigcup_{y\in \Pi _{\mathsf{P}_2}(x)}\Pi _{\mathsf{P}_1}(y)
    $$
for every $x\in G$. With respect to the composition \eqref{eq:product} the set of all correspondences (or, equivalently, the set of all polymorphisms) of $G$ is a semigroup, denoted by $\mathcal{P}(G)$, with involution $\mathsf{P}\mapsto \mathsf{P}^*$, identity element $P_1=\{(g,g),g \in G\}$ and zero element $P_0=G\times G$.

For later use we introduce also the higher powers $\mathsf{P}^n$ of $\mathsf{P}$, $n\ge2$, defined recursively by
    \begin{equation}
    \label{eq:powers}
\mathsf{P}^n=\mathsf{P}^{n-1}\star \mathsf{P}.
    \end{equation}

If $\mathsf{P}\subset G\times G$ is a correspondence such that the group homomorphisms $\pi_i\colon \mathsf{P}\longrightarrow G, i=1,2$, are injections, then $\mathsf{P}$ is (the graph of) an automorphism of $G$, and the conjugate correspondence yields the inverse automorphism. If $\pi_2$ is an injection then $\mathsf{P}$ is (the graph of) an $\emph{endomorphism}$ (i.e., of a surjective group homomorphism), and if $\pi_1$ is an injection then $\mathsf{P}$ is (the graph of) an \emph{exomorphism} (i.e., $\mathsf{P}^*$ is the graph of an endomorphism). \textit{The group of automorphisms as well as semigroups of endo- and exomorphisms are sub-semigroups of the semigroup of $\mathscr{P}(G)$ of correspondences of $G$.}

We note in passing that the product of the algebraic polymorphisms is a special case of the general notion
of the product of measure-preserving polymorphisms in \cite{Ve}.

    \begin{defi}
    \label{d:poly}
For the following definitions we fix a correspondence $\mathsf{P}$ of a compact group $G$. We write $\mathscr{B}_\mathsf{P}$ and $\lambda _\mathsf{P}$ for the Borel field and the Haar measure of $\mathsf{P}$.

\smallskip (1) \textit{Algebraic factor polymorphisms}. Let $H\subset G$ be a closed subgroup, and let $\mathsf{P}_{\negthinspace H}=\mathsf{P}/(H\times H)\subset (G/H\times G/H)$ be the associated \textit{factor correspondence}. The subgroup $H\subset G$ is \textit{invariant, co-invariant} or \textit{doubly invariant} under the polymorphism $\Pi _\mathsf{P}$ if $\mathsf{P}_{\negthinspace H}$ is an endomorphism, exomorphism or an automorphism. Examples will be given in Section \ref{s:examples}.

\smallskip (2) \textit{The Markov operator}. Put $\mathscr{B}_\mathsf{P}^{(i)}=\pi _i^{-1}(\mathscr{B}_G)\subset \mathscr{B}_\mathsf{P}, i=1,2$, and let $F_i\subset L^2(G,\mathscr{B}_G,\lambda _G)$ be the subspace of functions measurable with respect to  $\mathscr{B}_\mathsf{P}^{(i)},\,i=1,2$. Let $\mathrm{Pr}_i$ be the orthogonal projection in $L^2(G,\mathscr{B}_G,\lambda _G)$ onto $F_i, i=1,2$. We define the Markov operator
    $$
V_\mathsf{P}\colon L^2(G,\mathscr{B}_G,\lambda _G)\longrightarrow L^2(G,\mathscr{B}_G,\lambda _G)
    $$
as follows: if $f\in L^2(G,\mathscr{B}_G,\lambda _G)$, we define $h\in L^2(\mathsf{P},\mathscr{B}_\mathsf{P},\lambda _\mathsf{P})$ by
    $$
h(x,y)=f(x)
    $$
for every $(x,y)\in\mathsf{P}$ and set
    \begin{equation}
    \label{eq:VP}
V_\mathsf{P}f=E_{\lambda _\mathsf{P}}(h|\mathscr{B}_\mathsf{P}^{(2)}),
    \end{equation}
where $E_{\lambda _\mathsf{P}}(\cdot |\cdot )$ stands for conditional expectation with respect to $\lambda _\mathsf{P}$. Then
    \begin{equation}
    \label{eq:VP1}
V_\mathsf{P}=\mathrm{Pr}_2\cdot \mathrm{Pr}_1,\enspace V_\mathsf{P^*}=V^*_\mathsf{P}=\mathrm{Pr}_1 \cdot \mathrm{Pr}_2,
    \end{equation}
and
    \begin{equation}
    \label{eq:adjoint}
V_{\mathsf{P}^*}=V_\mathsf{P}^*
    \end{equation}
(cf. \eqref{eq:inverse}). Note that $V_\mathsf{P}$ preserves positivity and has norm $1$.

\smallskip (3) \textit{The Markov process $X_\mathsf{P}$}. The closed, shift-invariant subgroup
    \begin{equation}
    \label{eq:XP}
X_\mathsf{P} =\{(x_n)\in G^\mathbb{Z}:(x_n,x_{n+1})\in \mathsf{P} \enspace \textup{for every}\enspace n\in\mathbb{Z}\}
    \end{equation}
is the \emph{Markov process} of $\mathsf{P}$, and the corresponding \emph{Markov shift} $\sigma _\mathsf{P} \colon X_\mathsf{P} \longrightarrow X_\mathsf{P}$ is defined by $(\sigma _\mathsf{P} x)_n=x_{n+1}$ for every $x=(x_n)\in X_\mathsf{P} $. Note that $\sigma _\mathsf{P}$ is an automorphism of the compact group $X_\mathsf{P}$ which preserves the normalized Haar measure $\lambda _{X_\mathsf{P}}$ of $X_\mathsf{P}$, and that the Markov shift $\sigma _{\mathsf{P}^*} \colon X_{\mathsf{P}^*} \longrightarrow X_{\mathsf{P}^*}$ corresponding to $\mathsf{P}^*$ is the \emph{time reversal} of $\sigma _\mathsf{P}$.

Motivated by considering the various tail sigma-algebras (past, future and two-sided) of the Markov process $X_\mathsf{P}$ we call the polymorphism $\Pi _\mathsf{P}$ \textit{right} (\textit{left}, or \textit{totally}) \textit{nondeterministic} if there there is no closed invariant (co-invariant, or doubly invariant) proper subgroup $H\subset G$ (cf. Theorem \ref{t:markov}).

\smallskip (4) \textit{Ergodicity}. The polymorphism $\Pi _\mathsf{P}$ is \textit{ergodic} if the constants are the only $V_\mathsf{P}$-invariant functions.
    \end{defi}

    \begin{prop}
    \label{p:iso}
Let $G$ be a compact group and $\mathsf{P}\subset G\times G$ a correspondence. Then there exist closed normal subgroups $K_\mathsf{P}^{(i)}\subset G,\,i=1,2$, and a continuous group isomorphism $\eta _\mathsf{P}\colon G/K_\mathsf{P}^{(1)}\longrightarrow G/K_\mathsf{P}^{(2)}$ such that
    \begin{equation}
    \label{eq:eta}
\mathsf{P}=\{(g_1,g_2)\in G\times G:\eta _\mathsf{P}(g_1K_\mathsf{P}^{(1)})=g_2K_\mathsf{P}^{(2)}\}.
    \end{equation}
    \end{prop}

    \begin{proof}
We set $K_\mathsf{P}^{(1)}=\{g\in G:(g,1)\in \mathsf{P}\}$, $K_\mathsf{P}^{(2)}=\{g\in G:(1,g)\in \mathsf{P}\}$ and observe that $K_\mathsf{P}^{(1)}$ and $K_\mathsf{P}^{(2)}$ are normal subgroups of $G$, since $\pi _1(\mathsf{P})=\pi _2(\mathsf{P})=G$. Since $\mathsf{P}=\{(g_1p_1,g_2p_2):(g_1,g_2)\in \mathsf{P},\,p_1\in K_\mathsf{P}^{(1)},\,p_2\in K_\mathsf{P}^{(2)}\}$, we may view $\mathsf{P}$ as a subset $\bar{\mathsf{P}}\subset G/K_\mathsf{P}^{(1)}\times G/K_\mathsf{P}^{(2)}$, and the definition of the groups $K_\mathsf{P}^{(i)}$ implies that $\bar{\mathsf{P}}$ is the graph of a continuous group isomorphism $\eta _\mathsf{P}\colon G/K_\mathsf{P}^{(1)}\longrightarrow G/K_\mathsf{P}^{(2)}$.
    \end{proof}

    \begin{rema}
    \label{r:iso}
The triples $(K_\mathsf{P}^{(1)},K_\mathsf{P}^{(2)},\eta_\mathsf{P})$, where $K_\mathsf{P}^{(i)},\,i=1,2$, are subgroups of $G$ and $\eta _\mathsf{P}\colon G/K_\mathsf{P}^{(1)}\longrightarrow G/K_\mathsf{P}^{(2)}$ is a group isomorphism, form a parametrization of the algebraic polymorphisms of $G$.
    \end{rema}

    \begin{defi}
    \label{d:discrete}
A correspondence $\mathsf{P}\subset G\times G$ is \textit{finite-to-one} (and defines a \textit{polymorphism of discrete type}) if the groups $K_\mathsf{P}^{(i)}$ in \eqref{eq:eta} are both finite.

The finite-to-one correspondences of $G$ form a subsemigroup $\mathcal{P}_f(G)\subset \mathcal{P}(G)$ of the semigroup of all correspondences of $G$.
    \end{defi}

For the notation in the following characterization of (co-)invariance we again refer to \eqref{eq:eta}.

    \begin{theo}
    \label{t:factors}
Let $\mathsf{P}\subset G\times G$ be a correspondence and $H\subset G$ a closed normal subgroup.
    \begin{enumerate}
    \item
$H$ is invariant under the polymorphism $\Pi _\mathsf{P}$ if and only if $\eta _\mathsf{P}(H K_\mathsf{P}^{(1)})\subset H$;
    \item
$H$ is co-invariant under $\Pi _\mathsf{P}$ if and only if $\eta _\mathsf{P}^{-1}(HK_\mathsf{P}^{(2)})\subset H$;
    \item
$H$ is bi-invariant under $\Pi _\mathsf{P}$ if and only if $K_\mathsf{P}^{(1)}\subset H$ and $\eta _\mathsf{P}(H)=H$ \textup{(}in which case we also have that $K_\mathsf{P}^{(2)}\subset H$\textup{)}.
    \end{enumerate}
    \end{theo}

    \begin{proof}
Clearly, $K_\mathsf{P}^{(2)}\subset \eta _\mathsf{P}(HK_\mathsf{P}^{(1)})$. If $\eta _\mathsf{P}(HK_\mathsf{P}^{(1)})\subset H$ then invariance follows from Definition \ref{d:poly} (1). Conversely, if $K_\mathsf{P}^{(2)}\subset \eta _\mathsf{P}(HK_\mathsf{P}^{(1)})\subset H$, then $\mathsf{P}_{\negthinspace H}$ is the graph of a group endomorphism.

The other assertions are proved similarly.
    \end{proof}

    \begin{coro}
    \label{c:factors}
Let $\mathsf{P}\subset G\times G$ be a correspondence and let $H\subset G$ be a closed normal subgroup. We denote by $K_{\mathsf{P}^n}^{(i)},\,i=1,2$, the closed normal subgroups of $G$ associated with the correspondence $\mathsf{P}^n,\,n\ge2$, in \eqref{eq:powers} by Equation \eqref{eq:eta}. The sequences of subgroups $(K_{\mathsf{P}^n}^{(i)},\,n\ge1)$ are nondecreasing and have the following property.
    \begin{enumerate}
    \item
$H$ is invariant under $\Pi _\mathsf{P}$ if and only if it contains $\bigcup_{n\ge1}K_{\mathsf{P}^n}^{(2)}$;
    \item
$H$ is co-invariant under $\Pi _\mathsf{P}$ if and only if it contains $\bigcup_{n\ge1}K_{\mathsf{P}^n}^{(1)}$.
    \end{enumerate}
    \end{coro}

    \begin{proof}
If a closed normal subgroup $H\subset G$ is invariant under $\Pi _\mathsf{P}$ then Theorem \ref{t:factors} (1) shows that $K_{\mathsf{P}^2}^{(2)}=\eta _\mathsf{P}(K_\mathsf{P}^{(1)}K_\mathsf{P}^{(2)})\subset \eta _\mathsf{P}(K_\mathsf{P}^{(1)}H)\subset H$, hence $\eta _\mathsf{P}(K_{\mathsf{P}^2}^{(2)})\subset H$ and, by induction, $\eta _\mathsf{P}(K_{\mathsf{P}^n}^{(2)})\subset H$ for every $n\ge1$.

Conversely, if $H\supset \bigcup_{n\ge1}K_{\mathsf{P}^n}^{(2)}$, then it is obviously invariant.

The proof of the second assertion is analogous.
    \end{proof}

If the group $G$ is abelian, the characterization of ergodicity of a polymorphism of $G$ is completely analogous to that of ergodicity of an automorphism of $G$.

    \begin{theo}
    \label{t:ergodic}
Let $\mathsf{P}\subset G\times G$ be a correspondence of a compact abelian group $G$ with Markov operator $V_\mathsf{P}$ \textup{(}cf. \eqref{eq:VP}\textup{)}. Then $\Pi _\mathsf{P}$ is nonergodic if and only if there exist a nontrivial character $\chi $ of $G$ and an integer $n\ge1$ with $V_\mathsf{P}^n\chi =\chi $.
    \end{theo}

    \begin{proof}
If $\chi $ is a nontrivial character of $G$ then the restriction to $\mathsf{P}$ of $h=\chi \circ \pi _1$ is a nontrivial character on $\mathsf{P}$, and $V_\mathsf{P}\chi =E_{\lambda _\mathsf{P}}(h|\mathscr{B}_\mathsf{P}^{(2)})$ is either equal to zero or a nontrivial character of $G$ (depending on whether $h$ is constant on $\{1_G\}\times K_\mathsf{P}^{(2)}$ or not). Fourier expansion completes the proof of the theorem.
    \end{proof}

\section{Toral polymorphisms}

Compact groups do not have dynamically interesting polymorphisms unless they have large abelian quotients. For this reason we focus our attention in this section on compact \emph{abelian} groups, and in particular on finite-dimensional tori.

Let $m\ge1$, and let $\mathcal{P}_f(\mathbb{T}^m)$ be the semigroup of all finite-to-one correspondences of $\mathbb{T}^m$. We denote by $\mathcal{L}$ the semigroup of all finite index subgroups of $\mathbb{Z}^m$ with respect to the addition $L_1+L_2=\{u+v:u\in L_1,\,v\in L_2\}$. For every $\mathbf{n}=(n_1,\dots ,n_m)\in\mathbb{Z}^m$ and $x=(x_1,\dots ,x_m)\in\mathbb{T}^m$ we write
    \begin{equation}
    \label{eq:character}
\chi _\mathbf{n}(x)=e^{2\pi i\sum_{j=1}^mn_jx_j}
    \end{equation}
for the value of the corresponding character $\chi _\mathbf{n}$ of $\mathbb{T}^m$ at $x$. The annihilator of a subgroup $F\subset \mathbb{T}^m$ (or $F'\subset \mathbb{T}^{2m}$) is denoted by $F^\perp$ (resp. ${F'}^\perp$).

For $Q\in\textup{GL}(m,\mathbb{Q})$ we put
    \begin{equation}
    \label{eq:LQ}
\Lambda _Q=\mathbb{Z}^m\cap Q\mathbb{Z}^m\in\mathcal{L}.
    \end{equation}

Finally we introduce the semigroup
    \begin{equation}
    \label{eq:semigroup}
\mathcal{M}=\{(Q,\Lambda ):Q\in\textup{GL}(m,\mathbb{Q}),\,\Lambda \in\mathcal{L},\,\Lambda \subset \Lambda _Q\}
    \end{equation}
with composition
    \begin{equation}
    \label{eq:product2}
(Q,\Lambda )\cdot (Q',\Lambda ')= (QQ',\Lambda +Q\Lambda ').
    \end{equation}

    \begin{prop}
    \label{p:isomorphism}
The semigroup $\mathcal{P}_f(\mathbb{T}^m)$ is isomorphic to the semigroup $\mathcal{M}$ in \eqref{eq:semigroup}, where the isomorphism $\theta \colon \mathcal{M}\longrightarrow \mathcal{P}_f(\mathbb{T}^m)$ is given by
    \begin{equation}
    \label{eq:isomorphism}
\theta (Q,\Lambda )^\perp=\{(Q^{-1}\mathbf{n},\mathbf{n}): \mathbf{n}\in\Lambda \}
    \end{equation}
for every $(Q,\Lambda )\in\mathcal{M}$.

A correspondence $\mathsf{P}\in\mathcal{P}_f(\mathbb{T}^m)$ is connected if and only if
    \begin{equation}
    \label{eq:PM}
\mathsf{P}=\mathsf{P}_Q=\theta (Q,\Lambda _Q)
    \end{equation}
for some $Q\in\textup{GL}(m,\mathbb{Q})$ \textup{(}cf. \eqref{eq:LQ}\textup{)}. Finally, if $\mathsf{P}=\theta (Q,\Lambda )\in \mathcal{P}_f(\mathbb{T}^m)$, then $\mathsf{P}^*=\theta (Q^{-1},Q^{-1}\Lambda )$.
    \end{prop}

    \begin{proof}
For every $Q\in\textup{GL}(m,\mathbb{Q})$ and $\Lambda \subset \Lambda _Q$, $\theta (Q,\Lambda )\subset \mathbb{T}^m\times \mathbb{T}^m$ is obviously an element of $\mathcal{P}_f(\mathbb{T}^m)$, and \eqref{eq:eta} guarantees that every $\mathsf{P}\in\mathcal{P}_f(\mathbb{T}^m)$ is obtained in this manner.

If $\Lambda \subsetneq \Lambda _Q\subset \mathbb{Z}^m$, then $\mathsf{P}=\theta (Q,\Lambda )$ contains $\mathsf{P}_Q=\theta (Q,\Lambda _Q)$ as a finite index subgroup and is therefore not connected. In order to prove the converse we set
    $$
W_Q=\{(Q^{-1}\mathbf{n},\mathbf{n}):\mathbf{n}\in\Lambda _Q\}.
    $$
The dual group of $\mathsf{P}_Q$ is of the form $(\mathbb{Z}^m\times \mathbb{Z}^m)/W_Q$. If $\mathsf{P}_Q$ is not connected, then there exist an element $(\mathbf{m},\mathbf{n})\in(\mathbb{Z}^m\times \mathbb{Z}^m)\smallsetminus W_Q$ and an $l>1$ with $(l\mathbf{m},l\mathbf{n})\in W_Q$. Hence $(\mathbf{m},\mathbf{n})=(Q^{-1}\mathbf{k},\mathbf{k})$ for some $\mathbf{k}\in\mathbb{Z}^m\cap Q\mathbb{Z}^m=\Lambda _Q$, and $(\mathbf{m},\mathbf{n})\in W_Q$. This contradiction proves that $\mathsf{P}_Q$ is connected.

The last assertion is obvious.
    \end{proof}

    \begin{rema}
    \label{r:connected}
Proposition \ref{p:isomorphism} shows that connected finite-to-one correspondences are in one-to-one correspondence with the elements of $\textup{GL}(m,\mathbb{Q})$.
    \end{rema}

For every $n\ge1$ we define $\mathsf{P}^n$ and $K_{\mathsf{P}^n}^{(i)}$ as in Corollary \ref{c:factors}.

    \begin{theo}
    \label{t:factors2}
Let $Q\in\textup{GL}(m,\mathbb{R})$ and $\mathsf{P}=\xi (Q,\Lambda )\in\mathcal{P}_f(\mathbb{T}^m)$, where $\Lambda \subset \Lambda _Q=\mathbb{Z}^m\cap Q\mathbb{Z}^m$ is a finite index subgroup \textup{(}cf. \eqref{eq:LQ} and \eqref{eq:isomorphism}\textup{)}.
    \begin{enumerate}
    \item
The following conditions are equivalent.
    \begin{enumerate}
    \item
$\Pi _\mathsf{P}$ is right nondeterministic,
    \item
$\Xi _\mathsf{P}^+= \{\mathbf{n}\in\Lambda :Q^k\mathbf{n}\in\mathbb{Z}^m\enspace \textup{for every}\enspace k\le0\}=\{\mathbf{0}\}$,
    \item
$\bigcup_{n\ge1}K_{\mathsf{P}^n}^{(2)}$ is dense in $\mathbb{T}^m$.
    \end{enumerate}
    \item
The following conditions are equivalent.
    \begin{enumerate}
    \item
$\Pi _\mathsf{P}$ is left nondeterministic,
    \item
$\Xi _\mathsf{P}^- =\{\mathbf{n}\in\Lambda :Q^k\mathbf{n}\in\mathbb{Z}^m\enspace \textup{for every}\enspace k\ge0\}=\{\mathbf{0}\}$.
    \item
$\bigcup_{n\ge1}K_{\mathsf{P}^n}^{(1)}$ is dense in $\mathbb{T}^m$.
    \end{enumerate}
    \item
The following conditions are equivalent.
    \begin{enumerate}
    \item
$\Pi _\mathsf{P}$ is totally nondeterministic,
    \item
$\Xi _\mathsf{P}^+\cap \Xi _\mathsf{P}^-=\{\mathbf{n}\in\Lambda :Q^k\mathbf{n}\in \mathbb{Z}^m\enspace \textup{for every}\enspace k\in\mathbb{Z}\}=\{\mathbf{0}\}$.
    \item
Both $\bigcup_{n\ge1}K_{\mathsf{P}^n}^{(1)}$ and $\bigcup_{n\ge1}K_{\mathsf{P}^n}^{(2)}$ are dense in $\mathbb{T}^m$.
    \end{enumerate}
    \end{enumerate}
    \end{theo}

    \begin{proof}
In order to prove (1) we note that $\Xi _\mathsf{P}^+$ is a group and that $Q^{-1}\Xi _\mathsf{P}^+\subset \Xi _\mathsf{P}^+$. We set $H=(\Xi _\mathsf{P}^+)^\perp\subset \mathbb{T}^m$. Then the correspondence $\mathsf{P}_{\negthinspace H}=\mathsf{P}/(H\times H)$ is the graph of a continuous surjective homomorphism of the group $Y=\mathbb{T}^m/H^\perp$ to itself. The converse is proved by reversing this argument.

If the group $H=\bigcup_{n\ge1}K_{\mathsf{P}^n}^{(2)}$ is trivial, then $\mathsf{P}$ is the graph of an endomorphism. If $H$ is not dense in $\mathbb{T}^m$, then its closure $\bar{H}$ is nontrivial and is the smallest proper invariant subgroup of $\Pi _\mathsf{P}$ (cf. Corollary \ref{c:factors}).

The assertions (2) and (3) are proved in exactly the same manner.
    \end{proof}

The property of being left, right or totally nondeterministic can also be expressed in terms of the Markov group $X_\mathsf{P}$ in \eqref{eq:XP}.

    \begin{theo}
    \label{t:markov}
Under the hypotheses of Theorem \ref{t:factors2} the polymorphism $\Pi _\mathsf{P}$ is right nondeterministic if and only if the remote past of the process $X_\mathsf{P}$ is trivial.\footnote{The \textit{remote past} of the process $X_\mathsf{P}\subset (\mathbb{T}^m)^\mathbb{Z}$ is the intersection $\mathscr{A}_{-\infty }=\bigcap_{n\in\mathbb{Z}}\mathscr{A}_n^-$, where $\mathscr{A}_n^-$ is the sigma-algebra generated by the coordinates $-\infty ,\dots ,n$ of the process $X_\mathsf{P}$. The \textit{remote future} of $X_\mathsf{P}\subset (\mathbb{T}^m)^\mathbb{Z}$ is the sigma-algebra $\mathscr{A}_{\infty }=\bigcap_{n\in\mathbb{Z}}\mathscr{A}_n^+$, where $\mathscr{A}_n^+$ is generated by the coordinates $n,\dots ,\infty $ of $X_\mathsf{P}$.}

Similarly, $\Pi _\mathsf{P}$ is left nondeterministic if and only if the remote future of $X_\mathsf{P}$ is trivial.
    \end{theo}

    \begin{proof}
We denote by $\lambda _{X_\mathsf{P}}$ the Haar measure of the compact abelian group $X_\mathsf{P}$. For every $n\ge1$ we define the subgroups $K_{\mathsf{P}^n}^{(1)}$ and $K_{\mathsf{P}^n}^{(2)}$ as in Corollary \ref{c:factors}.

For the proof of the theorem it is enough to notice that the conditional measure $\lambda _{X_\mathsf{P}}(\,.\,|x_n=t)$ for fixed $t\in \mathbb{T}^m$ and $n<0$ is the uniform measure on a coset of of the group $K_{\mathsf{P}^n}^{(2)}$. According to Theorem \ref{t:factors} (1), the polymorphism $\Pi _\mathsf{P}$ is right nondeterministic if and only if $\bigcup_{n\ge 1}K_{\mathsf{P}^n}^{(2)}$ is dense in $\mathbb{T}^m$, in which case the conditional measures $\lambda _{X_\mathsf{P}}(\,.\,|x_{-n}=t)$ converge to $\lambda _{\mathbb{T}^m}$ as $m\to\infty $.
    \end{proof}

    \begin{theo}
    \label{t:ergodic2}
Let $Q\in\textup{GL}(m,\mathbb{R})$ and $\mathsf{P}=\xi (Q,\Lambda )\in\mathcal{P}_f(\mathbb{T}^m)$, where $\Lambda \subset \Lambda _Q=\mathbb{Z}^m\cap Q\mathbb{Z}^m$ is a finite index subgroup \textup{(}cf. \eqref{eq:isomorphism}\textup{)}. Then $\Pi _\mathsf{P}$ is nonergodic if and only if $Q$ has a nontrivial root of unity as an eigenvalue.

Furthermore, if $\Pi _\mathsf{P}$ is left, right or totally nondeterministic, then it is ergodic.
    \end{theo}

    \begin{proof}
By definition, $\mathsf{P}^\perp=\{(Q^{-1}\mathbf{n},\mathbf{n}): \mathbf{n}\in \Lambda \}$. A direct calculation shows that, for every $\mathbf{n}\in\mathbb{Z}^m$,
    \begin{equation}
    \label{eq:VP2}
V_\mathsf{P}(\chi _\mathbf{n})=
    \begin{cases}
\chi _{Q^{-1}\mathbf{n}}&\textup{if}\enspace \mathbf{n}\in\Lambda ,
    \\
0&\textup{otherwise},
    \end{cases}
    \enspace \enspace
V_\mathsf{P}^*(\chi _\mathbf{n})=
    \begin{cases}
\chi _{Q\mathbf{n}}&\textup{if}\enspace \mathbf{n}\in Q^{-1}\Lambda ,
    \\
0&\textup{otherwise},
    \end{cases}
    \end{equation}
(cf. \eqref{eq:character}). The existence of an $\mathbf{n}\in \Lambda $ with $V_\mathsf{P}^k\chi _\mathbf{n}=\chi _\mathbf{n}$ for some $k\ge1$ is obviously equivalent to $Q$ having a root of unity as an eigenvalue.

The last assertion is obvious.
    \end{proof}

We turn to the spectral properties of the Markov operator $V_\mathsf{P}$ associated with a correspondence $\mathsf{P}\in\mathcal{P}_f(\mathbb{T}^m)$.

    \begin{theo}
    \label{t:spectrum}
Let $m\ge1$, $\mathsf{P}\in\mathcal{P}_f(\mathbb{T}^m)$, and let $\textup{Sp}(V_\mathsf{P})\subset \mathbb{D}=\{z\in \mathbb{C}:|z|\le 1\}$ be the spectrum of the linear operator $V_\mathsf{P}\colon L^2_0(\mathbb{T}^m,\mathscr{B}_\mathbb{T}^m, \linebreak[0]\lambda _{\mathbb{T}^m})\longrightarrow L^2_0(\mathbb{T}^m,\mathscr{B}_\mathbb{T}^m,\lambda _{\mathbb{T}^m})$ in \eqref{eq:VP}, where $L^2_0(\mathbb{T}^m,\mathscr{B}_\mathbb{T}^m,\lambda _{\mathbb{T}^m})=\{f\in L^2(\mathbb{T}^m,\mathscr{B}_\mathbb{T}^m,\lambda _{\mathbb{T}^m}):\int f\,d\lambda _{\mathbb{T}^m}=0\}$ is  the orthocomplement of the constants.
    \begin{enumerate}
    \item
$\textup{Sp}(V_\mathsf{P})=\textup{Sp}(V_\mathsf{P}^*) =\{0\}$ if and only if $\mathsf{P}$ is totally nondeterministic;
    \item
$\textup{Sp}(V_\mathsf{P})=\textup{Sp}(V_\mathsf{P}^*) =\mathbb{S}=\{z\in\mathbb{C}:|z|=1\}$ if and only if $\Xi _\mathsf{P}^+=\Xi _\mathsf{P}^-=\Lambda $;
    \item
$\textup{Sp}(V_\mathsf{P})=\textup{Sp}(V_\mathsf{P}^*) \subset \mathbb{S}\cup\{0\}$ if and only if $\Xi _\mathsf{P}^+=\Xi _\mathsf{P}^-\subsetneq\Lambda $;
    \item
If $\Xi _\mathsf{P}^- \smallsetminus \Xi _\mathsf{P}^+\ne\varnothing $ then $\textup{Sp}(V_\mathsf{P})=\mathbb{D}$;
    \item
If $\Xi _\mathsf{P}^+ \smallsetminus \Xi _\mathsf{P}^-\ne\varnothing $ then $\textup{Sp}(V_\mathsf{P}^*)=\mathbb{D}$.
    \end{enumerate}
    \end{theo}

    \begin{proof}
We choose $(Q,\Lambda )\in\mathcal{M}$ with $\xi (Q,\Lambda )=\mathsf{P}$ (cf. \eqref{eq:isomorphism}). By definition of $\mathcal{M}$, $\Lambda \subset \Lambda \cap Q\Lambda $.

If $\mathsf{P}$ is totally nondeterministic then there exist, for every nonzero $\mathbf{n}\in\Lambda $, a smallest positive integer $k^+(\mathbf{n})$ and a largest negative integer $k^-(\mathbf{n})$ such that $Q^{k^\pm(\mathbf{n})}\mathbf{n}\notin\Lambda $. If we set
    $$
\mathcal{O}_Q(\mathbf{n})=\{Q^{k^-(\mathbf{n})+1} \mathbf{n},\dots ,Q^{k^+(\mathbf{n})-1}\mathbf{n}\},
    $$
then the restriction of $V_\mathsf{P}^*$ to the linear span $\langle \mathcal{O}_Q(\mathbf{n})\rangle $ of $\mathcal{O}_Q(\mathbf{n})$ in $L^2(\mathbb{T}^m, \mathscr{B}_{\mathbb{T}^m},\linebreak[0]\lambda _{\mathbb{T}^m})$ is unitarily equivalent to a matrix of the form
    $$
\left(
    \begin{smallmatrix}
0&1&0&0&\dots&0&0
    \\
0&0&1&0&\dots&0&0
    \\
\vdots&&&&\ddots&&\vdots
    \\
0&0&0&0&\dots&0&1
    \\
0&0&0&0&\dots&0&0
    \end{smallmatrix}
\right),
    $$
which has spectrum $\{0\}$. By taking the direct sum of the subspaces $\langle \mathcal{O}_Q(\mathbf{n})\rangle $, $\mathbf{n}\in\mathbb{Z}^n$, in $L^2(\mathbb{T}^m,\mathscr{B}_{\mathbb{T}^m},\lambda _{\mathbb{T}^m})$ we see that $\textup{Sp}(V_\mathsf{P}^*)=\{0\}$, and that the same is true for $\textup{Sp}(V_\mathsf{P})$. This proves (1).

The assertion (2) is obvious, since the condition given there is equivalent to $Q$ being an element of $\textup{GL}(m,\mathbb{Z})$.

In order to prove (3) we set $\Xi =\Xi _\mathsf{P}^+=\Xi _\mathsf{p}^-$, $S=\Xi ^\perp$, $Y=\widehat{\Xi }=\mathbb{T}^m/S$, and we observe that the restriction of $Q$ to $\Xi $ is a group automorphism. Hence the restrictions of $V_\mathsf{P}$ and $V_\mathsf{P}^*$ to the closed linear span of $\{\chi _\mathbf{n}:\mathbf{n}\in \Xi \}$ in $L^2(\mathbb{T}^m,\mathscr{B}_{\mathbb{T}^m},\lambda _{\mathbb{T}^m})$ are unitary.

For $\mathbf{n}\notin \Xi $ there exist a smallest nonnegative integer $k^+(\mathbf{n})$ and a largest nonpositive integer $k^-(\mathbf{n})$ such that $Q^{k^\pm(\mathbf{n})}\mathbf{n}\notin\mathbb{Z}^n$, and by combining the preceding paragraph with the argument in the proof of (1) we obtain (3).

If $\Xi _\mathsf{P}^-\smallsetminus \Xi _\mathsf{P}^+\ne\varnothing $ there exists an $\mathbf{n}\in\Lambda $ with $Q^k\mathbf{n}\in\Lambda $ for every $k\ge0$, but $Q^{-1}\mathbf{n}\notin\Lambda $. The restriction $W$ of $V_\mathsf{P}$ to the closed linear span $H$ of $\{\chi _{Q^k\mathbf{n}}:k\ge0\}$ has a nonzero kernel, since $V_\mathsf{P}\chi _\mathbf{n}=0$. Furthermore, if $\gamma \in\mathbb{C},\,|\gamma |<1$, and if $v_\gamma =\sum_{k\le0}\gamma ^k\chi _{Q^k\mathbf{n}}\in H$, then $V_\mathsf{P}v_\gamma =\gamma v_\gamma $, i.e., $v_\gamma $ is an eigenvector of $V_\mathsf{P}$ with eigenvalue $\gamma $. This proves that $\textup{Sp}(V_\mathsf{P})\supset \textup{Sp}(W)=\mathbb{D}$.

The same argument shows that $\textup{Sp}(V_\mathsf{P}^*)\supset \textup{Sp}(W)=\mathbb{D}$ if $\Xi _\mathsf{P}^+\smallsetminus \Xi _\mathsf{P}^-\ne\varnothing $, and the remaining implications are immediate consequences of what has already been shown.
    \end{proof}

\section{Factors of toral automorphisms and other examples}\label{s:examples}

If $A,B$ are endomorphisms of $\mathbb{T}^m$, then
    \begin{equation}
    \label{eq:PAB}
\mathsf{P}(A,B)=\{(x,y)\in \mathbb{T}^m\times \mathbb{T}^m:Ax=By\}
    \end{equation}
is a finite-to-one correspondence, and every $\mathsf{P}\in \mathcal{P}_f(\mathbb{T}^m)$ is of this form. Note that $\mathsf{P}(A,B)=\mathsf{P}(CA,CB)$ for every $C\in \textup{GL}(m,\mathbb{Z})$, and that $\mathsf{P}(AC',BC')$ are isomorphic if $C'\in\textup{GL}(m,\mathbb{Z})$.

The subgroups $K_{\mathsf{P}(A,B)}^{(1)}$ and $K_{\mathsf{P}(A,B)}^{(2)}$ and the isomorphism $\eta _{\mathsf{P}(A,B)}\colon \mathbb{T}^m/K_{\mathsf{P}(A,B)}^{(1)}\longrightarrow \mathbb{T}^m/K_{\mathsf{P}(A,B)}^{(2)}$ associated with $\mathsf{P}(A,B)$ by \eqref{eq:eta} are given by $K_{\mathsf{P}(A,B)}^{(1)}=(A^\top\mathbb{Z}^m)^\perp$, $K_{\mathsf{P}(A,B)}^{(2)}=(B^\top\mathbb{Z}^m)^\perp$ and $\eta _{\mathsf{P}(A,B)}=-B^{-1}A$, respectively, where ${}^\top$ denotes transpose.

Ergodicity of $\mathsf{P}(A,B)$ is thus equivalent to the assumption that $B^{-1}A\in\textup{GL}(m,\mathbb{Q})$ has no eigenvalues which are roots of unity (cf. Theorem \ref{t:ergodic}).

In order to characterize connectedness of $\mathsf{P}(A,B)$ we set $Q=\widehat{\eta _\mathsf{P}^{-1}}=B^\top (A^\top)^{-1}$. According to Proposition \ref{p:isomorphism}, $\mathsf{P}(A,B)$ is connected if and only if
    $$
B^\top\mathbb{Z}^m=\mathbb{Z}^m\cap B^\top (A^\top)^{-1}\mathbb{Z}^m=\Lambda _Q,
    $$
i.e., if and only if
    \begin{equation}
    \label{eq:intersection}
\mathbb{Z}^m=(A^\top )^{-1}\mathbb{Z}^m\cap (B^\top )^{-1}\mathbb{Z}^m.
    \end{equation}

    \begin{exam}
    \label{ex:poly1}
Let $m=1$, $k,l\in\mathbb{N}$ and $\mathsf{P}=\{(u,v): u,v\in \mathbb{T}, ku=lv\}$. We denote by $s$ the highest common factor of $k,l$ and set $k=k's$, $l=l's$. Then $\mathsf{P}=\theta (Q,\Lambda )$, where $\Lambda =l\mathbb{Z}$ and $Q$ is multiplication by $-\frac{k'}{l'}$ (cf. \eqref{eq:isomorphism}).

If $k,l$ are coprime (i.e., if $s=1$) then $\mathsf{P}=\mathsf{P}_Q$ is connected by \eqref{eq:intersection}.

If $s>1$ then $\phi $ is the quotient map from $\mathbb{Z}$ to $F=\mathbb{Z}/s \mathbb{Z}$, $\mathsf{P}$ is disconnected, and $K_\mathsf{P}^{(1)}\cap K_\mathsf{P}^{(2)}=\{x\in\mathbb{T}:sx=0\,(\textup{mod}\,1)\}$).

Finally, if $|k/l|\ne1$ then $\Pi _{\mathsf{P}}$ is ergodic. If $k,l$ are coprime and $|k|>1$, $|l|>1$, then $\Pi _\mathsf{P}$ is totally nondeterministic.
    \end{exam}

    \begin{exas}[Factors of polymorphisms]
    \label{e:factors}
(1) Consider the correspondence $\mathsf{P}=\{(u,v): u,v\in \mathbb{T}, 3u=2v\}$ (cf. Example \ref{ex:poly1} (1)), and let $H=\{0,1/5,2/5,3/5,4/5\}\subset \mathbb{T}$. Then $\mathsf{P}$ is the annihilator of $\{(3k,-2k):k\in\mathbb{Z}\}\subset \mathbb{Z}^2$ and $\mathsf{P}_{\negthinspace H}$ is the annihilator of $\{(15k,-10k):k\in\mathbb{Z}\}$. Note that $\mathsf{P}$ and $\mathsf{P}_{\negthinspace H}$ are isomorphic.

\medskip (2) Let $Q=\left(
    \begin{smallmatrix}
1&1\\1&0
    \end{smallmatrix}
\right)$, and let $\mathsf{P}=\mathsf{P}_Q=\theta (Q,\mathbb{Z}^2)$. Put $H=\left\{\mathbf{0},\left(\begin{smallmatrix}
1/2\\0
    \end{smallmatrix}\right)
\right\}\subset \mathbb{T}^2$ and set $\mathsf{P}'=\mathsf{P}_{\negthinspace H}$. We identify $\mathbb{T}^2/H$ with $\mathbb{T}^2$ by the map $\phi \left(\begin{smallmatrix}
s\\t
    \end{smallmatrix}\right)=\left(\begin{smallmatrix}
2s\\t
    \end{smallmatrix}\right)$
and view $\mathsf{P}'$ as a correspondence of $\mathbb{T}^2$. Then $\mathsf{P}'$ is isomorphic to the polymorphism of $\mathsf{P}(A,B)$ of $\mathbb{T}^2$ with
    $$
A=\left(
  \begin{matrix}
    1 & 0 \\
    0 & 2 \\
  \end{matrix}
\right),
\qquad
B=\left(
  \begin{matrix}
    1 & 2 \\
    1 & 0 \\
  \end{matrix}
\right).
    $$
Since $A,B\notin \textup{GL}(2,\mathbb{Z})$, $\mathsf{P}'$ is not the graph of an automorphism.
    \end{exas}

Example \ref{e:factors} (2) shows that a toral automorphism $\mathsf{P}$ may have a proper polymorphism as a factor (\emph{proper} means that the groups $K_\mathsf{P}^{(1)},K_\mathsf{P}^{(2)}$ in \eqref{eq:eta} are not both trivial). However, the following theorem shows that factors of automorphisms always have a nontrivial doubly invariant subgroup.

    \begin{theo}
    \label{t:deterministic}
Let $\mathsf{P}(A)\in\mathcal{P}_f(\mathbb{T}^m)$ be the graph of a toral automorphism $A\in\textup{GL}(m,\mathbb{Z})$. For every finite subgroup $H\subset \mathbb{T}^m$ there exists a finite doubly invariant subgroup $H'\subset \mathbb{T}^m$ containing $H$. In particular, $\mathsf{P}_{\negthinspace H'}$ is the graph of an automorphism of $\mathbb{T}^m/H'$.

In other words, if a polymorphism is a factor of a toral automorphism then it has a further factor which is again an automorphism.
    \end{theo}

    \begin{proof}
Since $H$ is finite, there exists a $q\ge1$ such that $H\subset H'=\{\mathbf{t}\in\mathbb{T}^m: q\mathbf{t}=\mathbf{0}\}$. As one checks easily, $\bigcup_{n\ge1}K_{\mathsf{P}^n}^{(i)}\subset H'$ for $i=1,2$. Theorem \ref{t:factors2} shows that $H'$ is invariant under $\mathsf{P}_{\negthinspace H}$, which proves our claim.
    \end{proof}

Theorem \ref{t:deterministic} allows us to say a little more about the structure of factors of toral automorphisms.

    \begin{coro}
    \label{c:structure}
Let $\mathsf{P}(A)\in\mathcal{P}_f(\mathbb{T}^m)$ be the graph of a toral automorphism $A\in \textup{GL}(m,\mathbb{Z})$, $H\subset \mathbb{T}^m$ a finite subgroup and $H'\subset \mathbb{T}^m$ a finite doubly invariant subgroup containing $H$.

If we identify both $\mathbb{T}^m/H$ and $\mathbb{T}^m/H'$ with $\mathbb{T}^m$, then the correspondence $\mathsf{P}''=\mathsf{P}_{\negthinspace H'}$ is the graph of a toral automorphism $A''$ \textup{(}i.e., $\mathsf{P}''=\mathsf{P}(A'')$\textup{)}, and the correspondence $\mathsf{P}'=\mathsf{P}_{\negthinspace H}$ has the graph of the automorphism $A''\in\textup{GL}(m,\mathbb{Z})$ as a factor with kernel $(H/H')\times (H/H')$.
    \end{coro}

    \begin{proof}
The identifications of $\mathbb{T}^m/H$ and $\mathbb{T}^m/H'$ with $\mathbb{T}^m$ yield finite-to-one equivariant homomorphisms
    $$
\mathbb{T}^m\longrightarrow \mathbb{T}^m/H\longrightarrow \mathbb{T}^m/H'
    $$
where the automorphisms $A$ and $A''$ act on the first and third torus and the polymorphism $\Pi _{\mathsf{P}'}$ on the second.
    \end{proof}

    \begin{rems}
    \label{r:structure}
(1) The automorphism $A''\in\textup{GL}(m,\mathbb{Z})$ in Corollary \ref{c:structure} is obviously conjugate to $A$ in $\textup{GL}(m,\mathbb{Q})$, but not necessarily in $\textup{GL}(m,\mathbb{Z})$.

Conversely, if $\mathsf{P}=\mathsf{P}(A)$ is the graph if some $A\in\textup{GL}(m,\mathbb{Z})$, and if $A''\in\textup{GL}(m,\mathbb{Z})$ is conjugate to $A$ in $\textup{GL}(m,\mathbb{Q})$, then the graph $\mathsf{P}(A'')$ is isomorphic to $\mathsf{P}(A)_{\negthinspace H'}$ for some finite subgroup $H'\subset \mathbb{T}^m$.

\medskip (2) There is a minimal choice of the subgroup $H'\subset \mathbb{T}^m$ in Theorem \ref{t:deterministic}: the subgroup generated by $\bigcup_{n\in\mathbb{Z}}A^k H$ (which we know to be finite from the proof of Theorem \ref{t:deterministic}).

\medskip (3) Corollary \ref{c:structure} shows that a polymorphism $\Pi _{\mathsf{P}'}$ is a factor of an automorphism $A\in \textup{GL}(m,\mathbb{Z})$ of $\mathbb{T}^m$ if and only if there exist an $A''\in\textup{GL}(m,\mathbb{Z})$ which is conjugate to $A$ in $\textup{GL}(m,\mathbb{Q})$ and finite groups $H\subset H'\subset \mathbb{T}^m$ such that $\mathsf{P}'$ is a skew product over the (the graph of) automorphism $A''$ with fibre $(H/H')$. Note, however, that $\mathsf{P}'$ is connected and is therefore a nontrivial $H/H'$-bundle over the base $\mathbb{T}^m$ on which $A''$ acts.

\medskip (4) In \cite{Ve} it is shown that every polymorphism is a factor of an automorphism with respect to some invariant partition (i.e., invariant sub-sigma-algebra), but Theorem \ref{t:deterministic} shows that this is not true in the algebraic category.
    \end{rems}

\end{document}